\documentclass[a4paper,10pt]{amsart}
\usepackage[margin=1.1in]{geometry}
\usepackage{amssymb}
\usepackage{amsmath}
\usepackage{amsthm}
\usepackage[hyperfootnotes=false]{hyperref} %\usepackage[hyperfootnotes=false,hidelinks]{hyperref}
\hypersetup{colorlinks=false}
\usepackage[utf8]{inputenc}

\usepackage{enumitem}

\usepackage{amsfonts}
\newtheorem{theorem}{Theorem}[section]
\newtheorem{lemma}[theorem]{Lemma}
\newtheorem{cor}[theorem]{Corollary}
\newtheorem{prop}[theorem]{Proposition}
\newtheorem{remark}{Remark}


\newcommand{\Addresses}{{% additional braces for segregating \footnotesize
  \bigskip
  \bigskip
  \footnotesize

  \textsc{MATHEMATISCH INSTITUUT, UNIVERSITEIT UTRECHT, POSTBUS 80.010, \\ 3508 TA UTRECHT, NEDERLAND}\par\nopagebreak
  \textit{E-mail}:  \texttt{l.p.grimmelt@uu.nl}

}}

\title{Goldbach Numbers in Short Intervals }
\date{}
\author{Lasse Grimmelt }

\begin{document}
\begin{abstract}
We decrease the length of the shortest interval for which almost all even integers in it are the sum of two primes. This is achieved by applying a new perspective of the circle method which builds on introducing a non-negative model for minorants. Compared to Harman's previous record, our approach uses minorants more efficiently and not only allows a shorter interval but also simplifies the treatment. 
\end{abstract}
\maketitle

\section{Introduction}
We believe that every even integer greater than $2$ is the sum of two primes, a Goldbach number. In spite of the fact that this conjecture itself has been open for over $250$ years, different approximations have been proved. Let $E(X)$ be the set of even integers up to $X$ that is \emph{not} the sum of two primes
\begin{align*}
E(X)=\{n\leq X: 2|n \text{ and }n\neq p_1+p_2 \text{ for all }p_1,p_2\in\mathbb{P}\}.
\end{align*}
Advancing the technique of Montgomery and Vaughan \cite{Imv}, Lu has shown \cite{lu} that
\begin{align*}
|E(X)|\ll X^{0.879}.
\end{align*}
We consider here the problem in short intervals and define
\begin{align*}
E(X,H)=\{n\in [X-H,X]: 2|n \text{ and }n\neq p_1+p_2 \text{ for all }p_1,p_2\in\mathbb{P}\}.
\end{align*}
Building up on the works of Jia, Li, Mikawa, Perelli, Pintz, and Ramachandra, Harman showed in chapter 10 of \cite{harb} that for any $\epsilon>0$ and $A>0$ in the range $H>X^{11/180+\epsilon}$ we have
\begin{align}\label{EHbound}
|E(X,H)|&\ll_{\epsilon,A} H(\log X)^{-A}.
\end{align}
The main result of this paper is the following theorem.
\begin{theorem}\label{MT}
Let $\epsilon,A>0$. The estimate \eqref{EHbound} holds for $H>X^{7/120+\epsilon}$.
\end{theorem}
Note that $11/180\approx 0.0611$ and $7/120\approx 0.0583$. The improvement is owing to the application of a version of the circle method that uses minorants more efficiently than Harman's approach. In particular, we prove no new results related to Dirichlet polynomials. From a technical point of view our strategy is quite simple. However, it uses a slightly different perspective of the circle method than previous applications of minorants. We motivate and explain it in detail in the remaining subsections of this introduction. 
	
\subsection*{Motivation and Further applications}

Even though we use neither Bohr sets nor the $W$-trick, our approach is motivated by the notion of pseudorandomness in the context of the transference circle method. See for example Green \cite{gre}, Green and Tao \cite{gt}, or Prendeville's survey \cite{0pre}. In particular - as in the transference principle - our main tool is the replacement of major and minor arc dissection by finding suitable model approximations in physical space that behaves nicely. In the transference circle method, the desired property of the model is being a bounded function, in this paper it is non-negativity. By skipping $W$-trick and Bohr sets, we achieve the same exceptional set bound \eqref{EHbound} as Harman.

The idea of a non-negative model as used in this paper can be applied to other problems in which a circle method that employs a point wise expential sum bounds and minorants. For example it allows for a simplified treatment of Vinogradov's Theorem with Chen primes, see Matomäki and Shao \cite{ms}. In similar spirit, this also gives a way to improve results on sums of four squares of almost twin primes, as considered by Tolev in \cite{to}.

\subsection*{Circle Method and Short Intervals}
For any functions $f_1, f_2: \mathbb{N}\to \mathbb{C}$ we denote throughout this paper their additive convolution by
\begin{align*}
f_1*f_2(n)=\sum_{n_1+n_2=n}f_1(n_1)f_2(n_2).
\end{align*}
In this notation we want to show that
\begin{align*}
\Lambda'*\Lambda'(n)>0
\end{align*}
for almost all even $n\in [X-H,X]$, where $\Lambda'$ is a weighted finitely supported prime indicator function
\begin{align*}
\Lambda'(n)=\begin{cases}\log n & \text{ if } n \text{ is prime and } n\leq X,\\
0 & \text{ else.}
\end{cases}
\end{align*}
We prove Theorem \ref{MT} by following a circle method approach based on Fourier analysis. For any arithmetic function $f:\mathbb{N}\to \mathbb{C}$ with finite support, we define the associated Fourier series $\hat{f}:\mathbb{T}\to \mathbb{C}$ by
\begin{align*}
\hat{f}(\alpha)=\sum_{n} f(n)e(\alpha n),
\end{align*}
where $e(z)=e^{2\pi \text{i} z}$. By orthogonality we can count the number of solutions we are interested in by
\begin{align*}
\Lambda'*\Lambda'(n)=\int_\mathbb{T} \widehat{\Lambda'}(\alpha)^2 e(-\alpha n)d\alpha.
\end{align*}
In the usual approaches, one now dissects $\mathbb{T}$ into major arcs $\mathfrak{M}$, minor arcs $\mathfrak{m}$ and applies different techniques on both parts. This is for example done in Harman's above mentioned result, see \cite{harb} chapter 10.3. The type of results that are used in his strategy are the same that we use to prove Theorem \ref{MT}. As mentioned before, our approach is based on a different way to interpret major and minor arc results. Instead of splitting up $\mathbb{T}$, we approximate the function we are interested in by a suitable model in physical space. This idea appears in similar form in chapter 19 of Iwaniec and Kowalski's book \cite{ik}.

Let us consider first what form this would take when bounding $E(X)$. Assume we have a model $\mathcal{T}_{\Lambda'}:\mathbb{N}\to \mathbb{R}$ that has the following two properties. First, it should approximate $\Lambda'$ on Fourier side by fulfilling for some $\theta>0$ the bound
\begin{align}\label{L'fcmodel}
\|\widehat{\Lambda'-\mathcal{T}_{\Lambda'}}\|_\infty \leq \theta X.
\end{align}
Second, we should be able to solve the counting problem after replacing one $\Lambda'$ by the model, i.e. understand
\begin{align*}
\Lambda'*\mathcal{T}_{\Lambda'}(n).
\end{align*}
We can use \eqref{L'fcmodel} to show that $\Lambda'*\Lambda'(n)$ is well approximated by $\Lambda'*\mathcal{T}_{\Lambda'}(n)$ for most $n\leq X$. Indeed, we have by the Bessel inequality
\begin{align*}
\sum_{n\leq X}\bigl|\Lambda'*\Lambda'(n)-\Lambda'*\mathcal{T}_{\Lambda'}(n)\bigr|^2&=\sum_{n\leq X}\bigl|\int_\mathbb{T}\widehat{\Lambda'}(\alpha)(\widehat{\Lambda'-\mathcal{T}_{\Lambda'}})(\alpha)e(-\alpha n)d\alpha\bigr|^2\\
&\leq \int_\mathbb{T}|\widehat{\Lambda'}(\alpha)|^2|\widehat{\Lambda'-\mathcal{T}_{\Lambda'}}(\alpha)|^2 d\alpha\\
&\leq \|\widehat{\Lambda'-\mathcal{T}_{\Lambda'}}\|_\infty^2 \|\Lambda'\|_2^2 \\
&\leq \theta^2 X^3 \log X.
\end{align*}
If $\theta=o((\log X)^{-1/2})$, this shows that for almost all $n\leq X$ we have
\begin{align*}
\Lambda'*\Lambda'(n)\approx \Lambda'*\mathcal{T}_{\Lambda'}(n).
\end{align*} 
Here the exact meaning of \emph{almost all} depends on the size of $\theta.$ The connection to major and minor arc strategies can be made by choosing
\begin{align}\label{majorarcapprox}
\mathcal{T}_{\Lambda'}(n)=\int_\mathfrak{M}\widehat{\Lambda'}(\alpha)e(-\alpha n)d\alpha.
\end{align}
Showing \eqref{L'fcmodel} then amounts to establishing minor arc bounds and $\Lambda'*\mathcal{T}_{\Lambda'}(n)$ can be understood by using major arc asymptotics of the associated Fourier series.

The use of Bessel's inequality is wasteful in the case of short intervals. To successfully bound $E(X,H)$, instead the following can be used
\begin{align*}
\sum_{n\in [X-H,X]}\bigl|\Lambda'*\Lambda'(n)-\Lambda'*\mathcal{T}_{\Lambda'}(n)\bigr|^2 \ll H X^{3/2} (\log X)^{3/2} \sup_\alpha \Bigl(\int_{-1/H}^{1/H}\bigl|\widehat{\Lambda'-\mathcal{T}_{\Lambda'}}(\alpha+\beta)\bigr|^2d\beta\Bigr)^{1/2}.
\end{align*}
If one now has an estimate of the form
\begin{align}\label{L'fcmodelshort}
\sup_\alpha \int_{-1/H}^{1/H}\bigl|\widehat{\Lambda'-\mathcal{T}_{\Lambda'}}(\alpha+\beta)\bigr|^2d\beta &\leq \theta X,
\end{align}
for some small enough $\theta$, one can deduce that
\begin{align*}
\Lambda'*\Lambda'(n)\approx \Lambda'*\mathcal{T}_{\Lambda'}(n)
\end{align*}
for almost all $n\in [X-H,X]$. A version of this approach can be found for example in Corollary 3.2 \cite{mrti}. We give the technical details in section 2. In this paper we get an exceptional set of size as stated in Theorem \ref{MT}, and work  with $\theta = (\log X)^{-A}$ for arbitrary large but fixed $A$.

To compare the estimate \eqref{L'fcmodelshort} to the previous condition \eqref{L'fcmodel}, it is helpful to apply Gallagher's Lemma (see Lemma \ref{GL}) to bound
\begin{align*}
\int_{-1/H}^{1/H}\bigl|\widehat{\Lambda'-\mathcal{T}_{\Lambda'}}(\alpha+\beta)\bigr|^2 d\beta \ll H^{-2} \sum_{t}\bigl| \sum_{n\in [t-H/2,t]}(\Lambda'-\mathcal{T}_{\Lambda'})(n)e(\alpha n)\bigr|^2.
\end{align*}
So one needs to replace the Fourier closeness of \eqref{L'fcmodel} by the stronger requirement that a similar condition holds on average in short intervals.

\subsection*{Minorants}
By following the strategy of the previous subsection one in fact shows an asymptotic evaluation for $\Lambda'*\Lambda'(n)$ for most $n\in [X-H,X]$. The range of $H$ that is admissible depends on our ability to show \eqref{L'fcmodelshort}. If we are only interested in lower bounds for $\Lambda'*\Lambda'(n)$, we can replace $\Lambda'$ by a suitable minorant 
\begin{align} \label{numinorant}
\Lambda'(n)\geq \nu(n)
\end{align} for which an analogue of \eqref{L'fcmodelshort} for $\nu$ and some model $\mathcal{T}_\nu$ can be proved for a smaller choice of $H$. Since $\Lambda'(n)\geq 0$, by \eqref{numinorant} and the Fourier closeness assumption \eqref{L'fcmodelshort} for $\nu$ and $\mathcal{T}_\nu$, we have for almost all $n\in [X-H,X]$ that
\begin{align*}
\Lambda'*\Lambda'(n)&\geq \Lambda'*\nu(n)\\
&\approx \Lambda'*\mathcal{T}_\nu(n).
\end{align*}
Here we use $\approx$ to denote the step in which the \emph{almost all} comes into play. This strategy does not require $\nu$ to be nonnegative and in praxis the constructed minorants take negative values. By choosing a suitable $\nu$, Jia \cite{jia} used this approach to prove a version of Theorem \ref{MT} for $H>X^{7/108+\epsilon}$. Note that $7/108\approx 0.0648$. Harman improved this result by observing that one is restricted by the need to understand
\begin{align*}
\Lambda'*\mathcal{T}_\nu(n)
\end{align*}
and that one can improve the choice of $\nu$ by replacing $\Lambda'$ by another minorant $\omega$. However, neither $\nu$ nor $\omega$ are nonnegative. Consequently one cannot bound $\Lambda'*\Lambda'(n)$ from below by $\omega*\nu(n)$. Instead, Harman uses additional majorants $\nu^+$ and $\omega^+$ and shows that one can still reach the desired lower bound by
\begin{align}\label{vsieveineq}
\Lambda'*\Lambda'(n)\geq (\nu*\omega^+ + \nu^+*\omega - \nu^+*\omega^+)(n),
\end{align}
see equation (10.1.4) \cite{harb}. This vector sieve type inequality appears in similar fashion in Brüdern and Fouvry's paper \cite{bf}. 

\subsection*{A Non-Negative Model}
The fact that the majorants appear with negative sign in \eqref{vsieveineq} means that to obtain nontrivial results the employed majorants and minorants need to be somewhat tight. The best range $H'$ for which we know that every interval $[X-H',X]$ contains almost the expected number of primes is 
\begin{align}\label{H'11/20}
H'\gg X^{11/20+\epsilon}.
\end{align}
More precisely, these intervals contain at least $0.99$ times the expected number of primes. See Theorem 10.3 \cite{harb} for the result used by Harman. Together with an upper bound of similar strength, this is numerically sufficient for the R.H.S of \eqref{vsieveineq} to give a nontrivial result, see equation (10.2.3) \cite{harb}. The condition \eqref{H'11/20} gives a lower bound for the support of $\nu$ and so appears as a factor in the admissible range for $H$. 

We remark that in their result on sums of three almost equal primes, Maynard, Matomäki, and Shao (\cite{mms}) also require numerically strong bounds. This causes the factor of $11/20$ in the exponent of their main theorem. However, the density requirement for them has a very different technical reason than in Harman's case.

In this paper we follow a different strategy that does not use \eqref{vsieveineq}. Though neither $\nu$ nor its classical major arc approximation are nonnegative, we can find a model that is.  In other words, there is a function
\begin{align}\label{posimodel}
\mathcal{T}^+_\nu(n):\mathbb{N}\to \mathbb{R}_{\geq 0}
\end{align} 
that is Fourier close to Harman's $\nu$ in the sense that 
\begin{align}\label{nuposimodel}
\sup_\alpha \int_{-1/H}^{1/H}\bigl|\widehat{\nu-\mathcal{T}_\nu^+}(\alpha+\beta)\bigr|^2d\beta \ll (\log X)^{-A} \|\nu\|_2^2.
\end{align}
So we can do the following steps. For almost all $n\in [X-H,X]$ we have
\begin{align}
\Lambda'*\Lambda'(n)\geq& \Lambda'*\nu(n) \label{strat1}\\
\approx& \Lambda'*\mathcal{T}^+_\nu(n)\label{strat2}\\
\geq& \omega*\mathcal{T}^+_\nu(n)\label{strat3}\\
\approx& \omega*\mathcal{T}_\nu(n)\label{strat4}.
\end{align}
Here \eqref{strat1} follows from $\Lambda'\geq 0$, \eqref{strat2} from the the Fourier closeness assumption \eqref{nuposimodel}, and \eqref{strat3} from $\mathcal{T}_\nu^+(n)\geq 0$, see \eqref{posimodel}. We explain \eqref{strat4} in the next subsection.  For this process to yield a non trivial result, besides standard major arc estimates, we only require $\nu$ and $\omega$ to have positive average. We can thus replace Harman's minorant by a $\omega$ for which we no longer need numerically strong results. The shortest interval length $H'$ for which we currently can prove the existence of primes in $[X-H',X]$ is
\begin{align*}
H'\gg X^{21/40},
\end{align*}
see Baker, Harman, and Pintz \cite{bhp}. Our saving over Harman's result in Theorem \ref{MT} is exactly the replacement of a factor of $11/20$ by $21/40$ in the exponent. Strictly speaking, we need more than just the existence of primes in the interval. We also require the aforementioned major arc estimates, i.e. knowledge about their distribution in arithmetic progressions to small moduli. However, the neccessary work is already done by Harman, see Theorem 10.8 \cite{harb}.

\subsection*{Choice of Model}
Eventhough the choice of $\mathcal{T}_\nu(n)$ and $\mathcal{T}_\nu^+(n)$ does not affect the result of Theorem \ref{MT}, it impacts technical aspects of the proof. It seems suitable to choose these model functions so that they are somewhat simple and the calculation of
\begin{align*}
\omega*\mathcal{T}_\nu(n)
\end{align*}
is as straight forward as possible. 

In Harman's construction the major arc main term of $\nu$ is, up to a scaling factor, exactly the same that we expect for the primes. The main term of the major arcs of order $Q$ for the weighted primes can be written as
\begin{align*}
\Lambda_Q(n):=\sum_{\substack{q\leq Q\\ a(q)^*}}\frac{\mu(q)}{\varphi(q)}e\bigl(\frac{an}{q}\bigr).
\end{align*}
In our case we have $Q=(\log X)^A$. For some $c_\nu>0$ that depends on the density of $\nu$, and for $n$ in the support of $\nu$ we set
\begin{align*}
\mathcal{T}_\nu(n)=c_\nu \Lambda_Q(n).
\end{align*}
As $\Lambda_Q$ is closely related to the major arcs, it can be used for solving additive problems in the primes without the circle method, see \cite{0hbv}. For its connection to sieves see  equation (11.31) of \cite{0ra} and equation (9.34) of \cite{cri}.

To solve $\omega*\mathcal{T}_\nu(n)$, we note that with this choice and suitably supported $\omega$ we have
\begin{align*}
\omega*\Lambda_Q(n)&=\sum_{n_1\leq n}\sum_{\substack{q\leq Q\\ a(q)^*}}\frac{\mu(q)}{\varphi(q)}e\bigl(\frac{a(n-n_1)}{q}\bigr)\omega(n_1)\\
&=\sum_{\substack{q\leq Q\\ a(q)^*}}\frac{\mu(q)}{\varphi(q)}e\bigl(\frac{an}{q} \bigr)\hat{\omega}(a/q).
\end{align*}
To calculate $\hat{\omega}(a/q)$ we can apply the aforementioned major arc information provided by Theorem 10.8 \cite{harb}. This leads to the usual singular series.

We now construct $\mathcal{T}_\nu^+(n)$. It has to be Fourier close to $c_\nu \Lambda_Q$ and so share the local distribution of primes for moduli up to $Q$. Since this local information is caused by the fact that the relevant primes have no nontrivial divisor less or equal to $Q$, the $Q$-rough numbers are an obvious choice for model. A $Q$-rough number is an integer in the support of $\rho(n,Q)$, where
\begin{align*}
\rho(n,Q)=\begin{cases}1 &\text{ if } p|n \rightarrow p> Q,\\
0 &\text{ else}
\end{cases}.
\end{align*}
Indeed, we have for $(a,q)=1$
\begin{align*}
\sum_{\substack{n\leq X \\ n\equiv a(q)}}\rho(n,Q)&=\frac{1}{\varphi(q)}\sum_{n\leq X}\rho(n,Q)\\
&\sim \frac{1}{\varphi(q)} \prod_{p\leq Q}(1-1/p) X.
\end{align*}
Together with a minor arc estimate, this was used by Cui, Li, and Xue \cite{clx} to give a lower bound to sums of sparse subsets of the primes. So we could choose $\mathcal{T}_\nu^+(n)$ as a scaled rough number indicator and generalize \cite{clx} into short intervals, as required for \eqref{nuposimodel}. 

However, we can slightly simplify the proof by using an upper bound sieve of sifting range $Q$. Since we can choose the level of distribution relatively large compared to $Q$, by a Fundamental Lemma we have similar major arc information. The necessary minor arc estimate is a simple Type I bound. For a suitable sieve $\theta_n$ we choose 
\begin{align*}
\mathcal{T}_\nu^+(n)=c_\nu \prod_{p\leq X}(1-1/p)^{-1} \theta_n
\end{align*}
and show that for $H$ as required for Theorem \ref{MT} we have
\begin{align*}
\sup_\alpha \int_{-1/H}^{1/H}\bigl|\widehat{\mathcal{T}_\nu^+-\mathcal{T}_\nu}(\alpha+\beta)\bigr|^2d\beta \ll (\log X)^{-A} \|\nu\|_2^2.
\end{align*}
In this way we get the $\approx$ in \eqref{strat4}.

\subsection*{Structure}

The structure of the paper is as follows. In the next section we formalise the above described process in the language of Fourier closeness. Afterwards, in section 3, we show how one can translate the short character sum bounds that Harman obtains for $\nu$ to relate it to $\Lambda_Q$. In section 4 we prove major arc asymptotic and minor arc bounds for a sieve in short intervals and so connect our nonnegative model to $\Lambda_Q$. Finally, in section 5 we import the required functions from Harman and prove Theorem \ref{MT}.

\section{The Circle Method with Minorants in Short Intervals}
We now prove prove the version of the circle method that we use for showing Theorem \ref{MT}. It makes precise the application of a nonnegative model as described in the introduction. The result is partially based on Corollary 3.2. \cite{mrti}, that is used by Matomäki, Radziwiłł, and Tao to obtain asymptotics for the representation of almost all integers in a short interval as the sum of two primes.

\begin{prop}[Circle Method] \label{prop}
Assume we are given parameters $X>0$, $H>0$, $0<\theta\leq 1$, and functions
\begin{align}
a, \mathcal{T}_\nu^+&:[X]\to \mathbb{R}_{\geq 0}\label{propeq-1}\\
b,\omega,\nu,\mathcal{T}_\nu&:[X]\to \mathbb{R}\nonumber
\end{align}
that fulfill the following conditions.

\begin{itemize}
\item \textsc{Minorization:} For all $n$ we have\begin{align}
\nu(n)&\leq b(n)\label{propeq1}\\
\omega(n)&\leq a(n).\label{propeq2}
\end{align}
\item \textsc{Fourier Closeness:} It holds that 
\begin{align}
\sup_\alpha \int_{-1/H}^{1/H} \bigl|\widehat{\nu-\mathcal{T}_\nu^+}(\alpha+\beta)\bigr|^2d\beta&\ll \theta \|\nu\|_2^2 \label{propeq3}\\
\sup_\alpha \int_{-1/H}^{1/H} \bigl|\widehat{\mathcal{T}_\nu-\mathcal{T}_\nu^+}(\alpha+\beta)\bigr|^2d\beta&\ll \theta \|\nu\|_2^2. \label{propeq4}
\end{align}
\end{itemize}
Then for any $\kappa >0$ and all $n\in [X-H,X]$ with at most 
\begin{align*}
O\bigl( \frac{H (\|\omega\|_2^2+\|a\|_2^2)(\|\nu\|_2^2+\|\mathcal{T}_\nu^+\|_2^2+\|\mathcal{T}_\nu\|_2^2) \sqrt{\theta}}{\kappa^2}\bigr)
\end{align*}
exceptions, we have
\begin{align*}
a*b(n)\geq \omega*\mathcal{T}_\nu(n)+O(\kappa).
\end{align*}
\end{prop}

\begin{proof}
We start by using the minorization property of $\nu$, see \eqref{propeq1}, together with the fact that $a$ by \eqref{propeq-1} only takes nonnegative values. For all $n$ it holds that
\begin{align*}
a*b(n)\geq a*\nu(n).
\end{align*}
We now show that $a*\nu(n)$ is well approximated by $a*\mathcal{T}_\nu^+(n)$. We  write the second moment of the difference as 
\begin{align*}
\sum_{n\in [X-H,X]}\bigl|a*(\nu-\mathcal{T}_\nu^+)(n) \bigr|^2&=\sum_{n\in [X-H,X]}\Bigl|\int_0^1 \widehat{a}(\alpha)\widehat{\nu-\mathcal{T}_\nu^+}(\alpha)e(-\alpha n) d\alpha\Bigr|^2.
\end{align*}
We can now apply almost verbatim part of the proof of Proposition 3.1 \cite{mrti}. There is a nonnegative Schwartz function $\Phi:\mathbb{R}\to \mathbb{R}_{\geq 0}$ such that $\Phi(x)\geq 1$ for $x\in [-1,1]$ and such that its Fourier transform $\hat{\Phi}(\xi)$ is supported in $[-1/2,1/2]$. We get
\begin{align*}
\sum_{n\in [X-H,X]}\Bigl|\int_0^1 \hat{a}(\alpha)\widehat{\nu-\mathcal{T}_\nu^+}(\alpha)e(-\alpha n) d\alpha\Bigr|^2&\leq \sum_{n\in [X-H,X]}\Bigl|\int_0^1\hat{a}(\alpha) \widehat{\nu-\mathcal{T}_\nu^+}(\alpha)e(-\alpha n) d\alpha\Bigr|^2 \Phi\Bigl(\frac{2(n-h_0)}{H}\Bigr),
\end{align*}
where $h_0=X-H/2$. By Poisson summation
\begin{align*}
\sum_{n}e((\alpha-\beta)n)\Phi\Bigl(\frac{2 n}{H}\Bigr)=O(H)
\end{align*}
and the expression vanishes, unless $\beta\in [\alpha-1/H,\alpha+1/H]$. We get
\begin{align*}
&\sum_{n\in [X-H,X]}\Bigl|\int_0^1\hat{a}(\alpha) \widehat{\nu-\mathcal{T}_\nu^+}(\alpha)e(-\alpha n) d\alpha\Bigr|^2 \Phi\Bigl(\frac{2(n-h_0)}{H}\Bigr)\\
\leq& \int_0^1 |\hat{a}(\alpha)\|\widehat{\nu-\mathcal{T}_\nu^+}(\alpha)|\int_0^1 |\hat{a}(\beta)\|\widehat{\nu-\mathcal{T}_\nu^+}(\beta)|\Bigl|\sum_{n}e((\alpha-\beta)n)\Phi\Bigl(\frac{2 (n-h_0)}{H}\Bigr) \Bigr|d\beta d\alpha\\
\ll& H \int_0^1 |\hat{a}(\alpha)\|\widehat{\nu-\mathcal{T}_\nu^+}(\alpha)|\int_{-1/H}^{1/H} |\hat{a}(\alpha+\beta)\|\widehat{\nu-\mathcal{T}_\nu^+}(\alpha+\beta)|d\beta d\alpha.
\end{align*}
By Cauchy's inequality and condition \eqref{propeq3} 
\begin{align*}
& H\int_0^1 |\hat{a}(\alpha)\|\widehat{\nu-\mathcal{T}_\nu^+}(\alpha)|\int_{-1/H}^{1/H} |\hat{a}(\alpha+\beta)\|\widehat{\nu-\mathcal{T}_\nu^+}(\alpha+\beta)|d\beta d\alpha\\
&\leq  H \|a\|_2^2 (\|\nu\|_2+\|\mathcal{T}_\nu^+\|_2) \sup_\alpha \Bigl(\int_{-1/H}^{1/H} \bigl|\widehat{\nu-\mathcal{T}_\nu^+}(\alpha+\beta)\bigr|^2d\beta\Bigr)^{1/2}\\
&\ll H\|a\|_2^2 (\|\nu\|_2^2+\|\mathcal{T}_\nu^+\|_2^2) \sqrt{\theta}.
\end{align*}
So we have
\begin{align*}
a*\nu(n)=a*\mathcal{T}_\nu^+(n)+O(\kappa)
\end{align*}
for all $n\in [X-H,X]$ with at most
\begin{align*}
O\bigl(\frac{H\|a\|_2^2 (\|\nu\|_2^2+\|\mathcal{T}_\nu^+\|_2^2) \sqrt{\theta}}{\kappa^2}\bigr)
\end{align*}
exceptions. By nonnegativity \eqref{propeq-1} and minorization \eqref{propeq2} we have for all $n$ that
\begin{align*}
a*\mathcal{T}_\nu^+(n)\geq \omega*\mathcal{T}_\nu^+(n).
\end{align*}
By the same argument as before, now using \eqref{propeq4}, we have
\begin{align*}
\omega*\mathcal{T}_\nu^+(n)=\omega*\mathcal{T}_\nu(n)+O(\kappa)
\end{align*}
for all $n\in [X-H,X]$ with at most
\begin{align*}
O\bigl(\frac{H\|\omega\|_2^2 |(\|\nu\|_2^2+\|\mathcal{T}_\nu\|_2^2) \sqrt{\theta}}{\kappa^2}\bigr)
\end{align*}
exceptions. The proposition follows by combining the two exceptional sets.
\end{proof}
We expect to apply the proposition in such way that we have an asymptotics of the form  \begin{align}
\omega*\mathcal{T}_\nu(n)\sim \mathfrak{S}(n)Y. \label{propeq5} 
\end{align}
Here $\mathfrak{S}(n)$ may encode local information and $Y$ is some parameter depending on the choices of $\omega$ and $\nu$.

\begin{remark}
If the loss incurred by $l^2$ bounds cannot be compensated by $\theta$, one may still reach a nontrivial result. To do so, one needs to modify the approach of the proposition to make use of large value estimates, see condition (iii) of Lemma 3.2 \cite{mrtii}.
\end{remark}

\section{From Short Character Sums to $\Lambda_Q$}
Condition \eqref{propeq3} of Proposition \ref{prop} asks for Fourier closeness between $\nu$ and $\mathcal{T}_\nu^+$. However, it is technically easier to first handle the case $\nu$ and $\mathcal{T}_\nu$. We deduce the bound \eqref{propeq3} at the end of the next section.  

As mentioned in the introduction, in our application the model $\mathcal{T}_\nu(n)$ is a suitable variant of $\Lambda_Q$. Recall that
\begin{align*}
\Lambda_Q(n)=\sum_{\substack{q\leq Q\\ a(q)^*}}\frac{\mu(q)}{\varphi(q)}e\bigl(\frac{an}{q}\bigr).
\end{align*}
This can be seen as the main term of the major arcs of the weighted primes and we show in this section how one extracts it from short character sum conditions as given in equation (10.2.8) of Theorem 10.2 in \cite{harb}.

We start by considering short exponential sums related to $\Lambda_Q$. 
\begin{lemma}\label{Lqshortnew}
Let $(r,q')=1$, $H'>0$, and $t>H'$. It holds that
\begin{align*}
\sum_{n\in [t-H',t]}\Lambda_Q(n)e\bigl(\frac{r n}{q'}\bigr)=\begin{cases} \frac{\mu(q') H'}{\varphi(q')}+O(Q^3) &\text{ if } q'\leq Q \\
O(q'Q+Q^3) &\text{ if } q'>Q.
\end{cases}
\end{align*}
\end{lemma}
Stronger results in the $Q$ aspects are possible, but this suffices for our application.
\begin{proof}
For $q'\leq Q$ we get
\begin{align*}
\sum_{n\in [t-H',t]}\Lambda_Q(n)e\bigl(\frac{r n}{q'}\bigr)&=\frac{\mu(q')}{\varphi(q')}H'+O\Bigl(\sum_{\substack{q\leq Q, a(q)^*\\ a/q\neq r/q'} }\frac{1}{\varphi(q)}\bigl|\sum_{n\in [t-H',t]}e\bigl((r/q'-a/q)n)\bigr|\Bigr).
\end{align*}
Since the set of fractions
\begin{align*}
\{a/q : q\leq Q, (a,q)=1 \}
\end{align*}
is $Q^{-2}$ separated, we have for $a/q\neq r/q'$ the estimate
\begin{align*}
\sum_{n\in [t-H',t]}e\bigl((r/q'-a/q)n)\ll Q^2.
\end{align*}
The case $q'\leq Q$ follows.

For $q'>Q$ we can start in the same manner and absorb the contribution of all fractions $a/q$ except the one for which
\begin{align*}
\bigl|\frac{r}{q'}-\frac{a}{q} \bigr|
\end{align*}
is minimal into the estimate $O(Q^3)$. For that one fraction we use that
\begin{align*}
\bigl|\frac{r}{q'}-\frac{a}{q} \bigr|\gg \frac{1}{q'Q}.
\end{align*}
The bound
\begin{align*}
\sum_{n\in [t-H',t]}e\bigl((\frac{r}{q'}-\frac{a}{q})n\bigr)\ll q'Q
\end{align*}
completes the proof of the lemma.
\end{proof}

We also require Gallagher's Lemma.
\begin{lemma}[Gallagher]\label{GL}
Let $Z_1\leq Z_2$ and $f(n)$ be a sequence of complex numbers supported in $[Z_1,Z_2]$. Let further $2<\Delta<(Z_2-Z_1)/2$. We have
\begin{align*}
\int_{-1/\Delta}^{1/\Delta}|\hat{f}(\beta)|^2d\beta \ll \Delta^{-2}\sum_{t}\bigl|\sum_{n\in [t-\Delta/2,t]}f(n) \bigr|^2
\end{align*}
\end{lemma}
\begin{proof}
This is Lemma 1 in \cite{gal}.
\end{proof}

We now define $\mathcal{T}_\nu$. For parameters $Y$, $c_\nu$, a large $A>0$, we choose $Q=(\log X)^A$ and set for $n\in (Y,2Y]$ 
\begin{align}
\mathcal{T}_\nu(n)&=c_\nu \Lambda_Q(n) \label{Tnudef}\\
&=c_\nu \sum_{\substack{q\leq Q\\ a(q)^*}}\frac{\mu(q)}{\varphi(q)}e\bigl(\frac{a n}{q} \bigr)\nonumber.
\end{align}
Further $\mathcal{T}_\nu(n)=0$ for any other $n$. With this we can state the main result of this section. 

\begin{lemma}[From short character sums to to $\Lambda_Q$]\label{shorttoLq}
Let $A>0$, $\delta>0$, $c_\nu>0$, $Y>0$ and $Y^{\delta}<H<Y^{1-\delta}$ be parameters. Let further $\nu(n):(Y,2Y]\cap \mathbb{Z}\to \mathbb{R}$ be a function that is supported on $H^{1/2}$-rough numbers only and fulfills for some $B\leq A-1$ the bound
\begin{align}\label{nuest}
\sum_n |\nu(n)|^2 \ll Y (\log Y)^B.
\end{align}
Assume that for any $q\leq H^{1/2}$ it holds that
\begin{align}\label{nushort1}
\sum_{\substack{\chi(q)\\ \chi\not \in E_q}}\sum_t \Bigl|\sum_{n\in [t-qH^{1/2}/3,t]}(\nu(n)\chi(n)-\delta_\chi c_\nu)\Bigr|^2 dy\ll q^2 H Y (\log Y)^{-A-1}.
\end{align}
Here $\delta_\chi=1$ or $\delta_\chi=0$ depending on whether $\chi$ is principal or not and $E_q$ is a set of $O(q^{1/2}(\log Y)^{-2A})$ characters.

Then, for $\mathcal{T}_\nu$ as defined in \eqref{Tnudef}, we have
\begin{align*}
\sup_\alpha \int_{-1/H}^{1/H}\bigl|\widehat{\nu-\mathcal{T}_\nu} (\alpha+\beta)\bigr|^2d\beta \ll  Y(\log Y)^{-A}.
\end{align*}
\end{lemma}

\begin{proof}
The proof is based on section 10.3 of \cite{harb}. Dissect $[0,1]$ into Farey arcs of order $\lfloor H^{1/2} \rfloor$ and write $I_{q,r}$ for the arc centered at $r/q$. Then
\begin{align}
I_{q,r}\subset \Bigl[\frac{r}{q}-\frac{1}{q\lfloor H^{1/2}\rfloor},\frac{r}{q}+\frac{1}{q \lfloor H^{1/2}\rfloor } \Bigr] \label{Iqrrange}
\end{align}
and for any fixed $\alpha$ at most $2$ intervals $I_{q,r}$ overlap $[\alpha-1/H,\alpha+1/H]$. We get
\begin{align*}
\sup_\alpha \int_{-1/H}^{1/H}\bigl|\widehat{\nu-\mathcal{T}_\nu} (\alpha+\beta)\bigr|^2\ll \max_{\substack{q\leq H^{1/2}\\r(q)^*}} \int_{I_{q,r}} \bigl|\widehat{\nu-\mathcal{T}_\nu}(\beta)\bigr|^2d\beta.
\end{align*}
We use \eqref{Iqrrange} and apply Lemma \ref{GL} to get
\begin{align}\label{proofeq1}
\int_{I_{q,r}} \bigl|\widehat{\nu-\mathcal{T}_\nu}(\beta)\bigr|^2d\beta\ll \frac{1}{q^2 H}\sum_t \bigl|\sum_{n\in [t-qH^{1/2}/3,t]}(\nu(n)-\mathcal{T}_\nu(n))e\bigl(\frac{rn}{q}\bigr) \bigr|^2.
\end{align}
We recall the choice of $\mathcal{T}_\nu$ given in \eqref{Tnudef}. By Lemma \ref{Lqshortnew} we get 
\begin{align*}
\sum_{n\in [t-qH^{1/2}/3,t]}\mathcal{T}_\nu(n)e\bigl(\frac{rn}{q}\bigr)=\begin{cases}
\frac{\mu(q)}{\varphi(q)}c_\nu\sum_{n\in [t-qH^{1/2}/3,t]\cap (Y,2Y]}1+O(Q^3) &\text{ if } q\leq Q \\
O(qQ+Q^3) &\text{ if } Q<q\leq H^{1/2}.
\end{cases}
\end{align*}

To consider $\nu$, we use the following Ramanujan sum and some of its basic properties. For a Dirichlet character $\chi$ to the modulus $q$ we write
\begin{align*}
\tau(\chi)=\sum_{a(q)^*}\chi(r)e\bigl(\frac{r}{q}\bigr).
\end{align*}
It holds that
\begin{align}
|\tau(\chi)|&\leq \sqrt{q} \label{rama1}\\
\tau(\chi_0)&=\mu(q), \label{rama2}
\end{align}
where $\chi_0$ is the principal character. Ramanujan sums enter our proof by observing that for $(rn,q)=1$ we have
\begin{align}
e\bigl(\frac{r n}{q})=\frac{1}{\varphi(q)}\sum_{\chi(q)}\tau(\overline{\chi})\chi(r n) \label{rama3}.
\end{align}
Using the fact that $\nu$ is supported on $H^{1/2}$-rough numbers only and $q\leq H^{1/2}$, we apply \eqref{rama3} and \eqref{rama2} to obtain
\begin{align*}
\sum_{n\in [t-qH^{1/2}/3,t]} \nu(n)e\bigl(\frac{rn}{q}\bigl)&=\sum_{n\in [t-qH^{1/2}/3,t]}\nu(n)\frac{1}{\varphi(q)}\sum_{\chi(q)}\tau(\overline{\chi})\chi(n r)\\
&=\frac{1}{\varphi(q)}\sum_\chi\tau(\overline{\chi})\chi(r)\sum_{n\in [t-qH^{1/2}/3,t]}\nu(n)\chi(n)\\
&=\frac{\mu(q)}{\varphi(q)}c_\nu\sum_{n\in [t-qH^{1/2}/3,t]}1+E(t,q,r),
\end{align*}
where
\begin{align*}
E(t,q,r)=\sum_{\chi(q)} \frac{\tau(\overline{\chi}(q))\chi(r)}{\varphi(q)}\sum_{n\in [t-qH^{1/2}/3,t]}\bigl(\nu(n)\chi(n)-\delta_\chi c_\nu\bigr).
\end{align*}

We now return to \eqref{proofeq1}. Let us first consider the case $q\leq Q$. By above considerations, if $t\in (Y+qH^{1/2}/3,Y]$ then the main terms in
\begin{align*}
\sum_{n\in [t-qH^{1/2}/3,t]}(\nu(n)-\mathcal{T}_\nu(n))e\bigl(\frac{rn}{q}\bigr) 
\end{align*}
cancel exactly. For other $t$ we can bound the main term contribution of $\nu$ and $\mathcal{T}_\nu$ to \eqref{proofeq1} by $H \log H$. Similarly, if $q>Q$ we can bound the main term of $\nu$ trivially by using the estimate $\phi(q)\gg q/\log \log q$. By including all error terms, we obtain for any $q\leq H^{1/2}$ the evaluation
\begin{align*}
\frac{1}{q^2 H}\sum_t \bigl|\sum_{n\in [t-qH^{1/2}/3,t]}(\nu(n)-\mathcal{T}_\nu(n))e\bigl(\frac{rn}{q}\bigr) \bigr|^2&= \frac{1}{q^2 H}\sum_t \bigl|E(t,q,r) \bigr|^2\\&+O\Bigl(Y\bigl(\frac{ (\log \log H)^2}{Q^2}+\frac{Q^2}{H}+\frac{Q^3}{H}\bigr)+H\log H\Bigr).
\end{align*}
Since $Y^\delta<H<Y^{1-\delta}$ and $Q=(\log X)^A$, the $O$ term is sufficiently small and we are left with bounding the contribution of $E$.

We split up the sum over $\chi$ in the definition of $E(t,q,r)$ into whether $\chi \in E_q$ or not
\begin{align*}
E(t,q,r)&=\sum_{\chi\in E_q}+\sum_{\chi \not \in E_q}\\
&=E_1(t,q,r)+E_2(t,q,r),
\end{align*}
say. The Lemma now follows, if for $i \in \{1,2\}$ we can show
\begin{align*}
\frac{1}{q^2 H}\sum_{t}\bigl|E_i(t,q,r)|^2 \ll  Y (\log Y)^{-A}.
\end{align*}
By Cauchy's inequality, condition \eqref{nuest}, the estimate for $|\tau(\chi)|$ given by \eqref{rama1}, and the bound for $E_q$ we have
\begin{align*}
\frac{1}{q^2 H}\sum_{t}\bigl|E_1(t,q,r)|^2&\leq \frac{q |E_q|}{\varphi(q)^2}(Y+\sum_{n}|\nu(n)|^2) \\
&\ll Y (\log Y)^{B-2A+1}\\
&\ll Y (\log Y)^{-A}.
\end{align*}
Similarly, for $i=2$ we again use Cauchy's inequality to get
\begin{align*}
\frac{1}{q^2 H}\sum_{t}\bigl|E_2(t,q,r)|^2&\leq \frac{1}{\varphi(q)qH} \sum_{\substack{\chi(q)\\ \chi \not \in E_q}}\sum_{t} \bigl|\sum_{n\in [t-qH^{1/2}/3,t]}\bigl(\nu(n)\chi(n)-\delta_\chi c_\nu\bigr)\bigr|^2.
\end{align*}
By the assumed bound \eqref{nushort1}, this completes the proof.
\end{proof}

\section{From a Sieve to $\Lambda_Q$}
We now show that a suitably scaled upper bound sieve is Fourier close to $\Lambda_Q$. Since we intend to work with a Fundamental Lemma, the precise choice of the sieve is not important. We work with the $\beta$-sieve as described in chapter 6.4 of Opera de Cribro \cite{cri}. Let 
\begin{align*}
\theta_n=\theta_n(D,z)
\end{align*}
be the upper bound $\beta$-sieve for $\beta=10$ with sifting range $z$ and level $D$. Note that this means in particular $\theta_n\geq 0$. We write
\begin{align}\label{sievconv}
\theta_n=\sum_{d|n}\lambda_d,
\end{align} 
where $\lambda_d$ are the associated sieve weights that fulfill $|\lambda_d|\leq 1$ and are supported on $d\leq D$ only. 

We now show an analogue result of Lemma \ref{Lqshortnew} for these sieve weights. 

\begin{lemma}\label{Sieveshortnew}
Let $(r,q)=1$, $H'>0$, $t>H'$, and $\theta_n=\theta_n(D,z)$ as given above. We have
\begin{align*}
\prod_{p\leq z}\bigl(1-p^{-1})^{-1}\sum_{n\in [t-H',t]}\theta_ne\bigl(\frac{rn}{q}\bigr)=\begin{cases} \frac{\mu(q) H'}{\varphi(q)}+O(H'e^{-\frac{\log D}{\log z}}+qD) &\text{ if } q\leq z \\
O\bigl((H'/q+D+q)\log qH'\bigr) &\text{ if } q>z.
\end{cases}
\end{align*}
\end{lemma}
\begin{proof}
Let first $q\leq z$. We sort into residues classes mod $q$ to get
\begin{align*}
\sum_{n\in [t-H',t]}\theta_ne\bigl(\frac{rn}{q}\bigr)&=\sum_{a(q)}e\bigl(\frac{r a}{q}\bigr)\sum_{\substack{n\in [t-H',t]\\ n\equiv a(q)}}\theta_n.
\end{align*}
Now
\begin{align*}
\sum_{\substack{n\in [t-H',t]\\ d|n \text{ and }n\equiv a(q)}}1=\frac{g(d,q,a)H'}{q}+O(1),
\end{align*}
where $g(d,q,a)$ is multiplicative in $d$ and for primes given by
\begin{align*}
g(p,q,a)=\begin{cases}\frac{1}{p} &\text{ if } p\nmid q \\
1 &\text{ if } p|q \text{ and } p|  a \\
0 &\text{ else.}
\end{cases}
\end{align*}
By Lemma 6.8 \cite{cri}, the trivial bound $V^+(D,z)\geq V(z)$, and a short calculation involving the sifting density function $g(d,q,a)$, we have for $(a,q)=1$
\begin{align*}
\prod_{p\leq z}\bigl(1-p^{-1})^{-1}\sum_{\substack{n\in [t-H',t]\\ n\equiv a(q)}}\theta_n=\frac{H'}{\varphi(q)}\bigl(1+O(e^{-\frac{\log D}{\log z}})\bigr)+O(D).
\end{align*}
We cannot apply Lemma 6.8 \cite{cri} directly on the case $(a,q)\neq 1$. However, by using the case $q=1$ and subtracting the contribution of coprime values, we can bound
\begin{align*}
\prod_{p\leq z}\bigl(1-p^{-1})^{-1}\sum_{\substack{a(q)\\ (a,q)\neq 1}}\sum_{\substack{n\in [t-H',t]\\ n\equiv a(q)}}\theta_n=O(H'e^{-\frac{\log D}{\log z}}+qD).
\end{align*}
Thus we get 
\begin{align*}
\prod_{p\leq z}\bigl(1-p^{-1})^{-1}\sum_{n\in [t-H',t]}\theta_ne\bigl(\frac{rn}{q}\bigr)=\frac{\mu(q)H'}{\varphi(q)}+O(H'e^{-\frac{\log D}{\log z}}+ q D)
\end{align*}
and have completed the $q\leq z$ case of the lemma.

For $q>z$ we use the convolution expression \eqref{sievconv} to identify the Type I bound 
\begin{align*}
\bigr|\sum_{n\in [t-H',t]}\theta_n e\bigl(\frac{rn}{q}\bigr)\bigl|&\leq \sum_{d\leq D} \bigl|\sum_{n \in[(t-H')/d,t/d] }e\bigl(\frac{rnd}{q}\bigr)\bigr|\\
&\leq \sum_{d\leq D}\min \bigl\{\frac{H'}{d}+1,\|\frac{rd}{q}\|^{-1}\bigr \}.
\end{align*}
As described in chapter 13.5 \cite{ik} we can estimate
\begin{align*}
\sum_{d\leq D}\min \{\frac{H'}{d}+1,\|\frac{rd}{q}\|^{-1}\}\ll (D+\frac{H'}{q}+q)\log 2qH'.
\end{align*}
This completes the proof.
\end{proof}

We now choose $\mathcal{T}_\nu^+$. Recall that $Q=(\log Y)^A$. For still unspecified parameter $Y$, $c_\nu$, and $H$ we set $z=Q$, $D=H^{1/10}$ and for $n\in (Y,2Y]$
\begin{align*}
\mathcal{T}_\nu^+(n)=c_\nu \prod_{p\leq z}\bigl(1-1/p)^{-1} \theta_n(D,z).
\end{align*}
Further $\mathcal{T}_\nu^+(n)=0$ for other $n$. Since $\theta_n$ is an upper bound sieve, surely
\begin{align*}
\mathcal{T}_\nu^+(n)\geq 0.
\end{align*}

\begin{lemma}[From a Sieve to $\Lambda_Q$]\label{sievetolambdaq}
Let $\delta>0$ and $Y^\delta<H<Y^{1-\delta}$. It holds that
\begin{align*}
\sup_\alpha \int_{-1/H}^{1/H} \bigl|\widehat{\mathcal{T}_\nu-\mathcal{T}_\nu^+} (\alpha+\beta)\bigr|^2d\beta \ll  Y(\log Y)^{-A}.
\end{align*}

\end{lemma}
\begin{proof}
We use the same Farey arc dissection as in the proof of Lemma \ref{shorttoLq}. It thus suffices similar as before to bound
\begin{align*}
\max_{\substack{q\leq H^{1/2}\\ r(q)^*}}\sum_t \frac{1}{q^2 H}\bigl|\sum_{n\in [t-qH^{1/2}/3,t]}\bigl(\mathcal{T}_\nu(n)-\mathcal{T}_\nu^+(n)\bigl)e(\frac{rn}{q}) \bigr|^2.
\end{align*}
By the range of support, we can restrict the summation to $t\in (Y,2Y+qH^{1/2}/3]$. 
We use Lemma \ref{Lqshortnew} and \ref{Sieveshortnew} to obtain  
\begin{align*}
\bigl|\sum_{n\in [t-qH^{1/2}/3,t]}\bigl(\mathcal{T}_\nu(n)-\mathcal{T}_\nu^+(n)\bigl)e(\frac{rn}{q}) \bigr|=O\bigl((qQ+Q^3+qH^{1/2}e^{-\frac{\log D}{\log z}}+qH^{1/10}+\frac{q H^{1/2}}{Q}+q) \log Y\bigr)
\end{align*} 
By the choice of parameters 
\begin{align*}
z&=Q=(\log Y)^A\\
Y^\delta&<H<Y^{1-\delta}\\
D&=H^{1/10}
\end{align*}
and range of $q$ to consider, this means 
\begin{align*}
\bigl|\sum_{n\in [t-qH^{1/2}/3,t]}\bigl(\mathcal{T}_\nu(n)-\mathcal{T}_\nu^+(n)\bigl)e(\frac{rn}{q}) \bigr|=O\bigl(\frac{H^{1/2}q}{(\log Y)^{A-1}}\bigr)
\end{align*}
and so completes the proof.
\end{proof}

With this, we can deduce quickly that $\nu$ is Fourier close to $\mathcal{T}_\nu^+$.
\begin{cor}\label{cor}
Assume the condition of Lemma \ref{shorttoLq}. It holds that
\begin{align*}
\sup_\alpha \int_{-1/H}^{1/H} \bigl|\widehat{\nu-\mathcal{T}_{\nu}^+}(\alpha+\beta
)\bigr|^2 d\beta \ll Y(\log Y)^{-A}.
\end{align*}
\end{cor}
\begin{proof}
The statement follows from the Lemma \ref{shorttoLq}, Lemma \ref{sievetolambdaq}, and the triangle inequality.
\end{proof}

\section{Proof of Theorem \ref{MT}}
The proof of our main theorem is done by applying Proposition \ref{prop} with suitable functions $\nu$ and $\omega$. Here the function $\nu$ fulfills the conditions of Lemma \ref{shorttoLq}. The functions  are based on constructions by Harman, we summarize their properties in the following Lemma. We recall that $A>0$ and $Q=(\log Y)^A$. 

\begin{lemma}\label{importlemma}
Let $\epsilon>0$ and $X$ be sufficiently large. Set $Y=X^{21/40+\epsilon}$ and $H=Y^{1/9+2\epsilon}$. There exist functions $\nu,\omega: \mathbb{N}\to \mathbb{R}$ with the following properties. 
\begin{itemize}
\item $\nu$ is supported on $(Y,2Y]$ and $\omega$ on $(X-3Y,X-Y]$.
\item We have $\nu(n)\leq \Lambda'(n)$ and $\omega\leq \Lambda'(n)$ for all $n$.
\item $\nu$ fulfills the conditions of Lemma \ref{shorttoLq} for some $c_\nu>0.05$.
\item $\omega(n)$ is supported on $Q$-rough numbers only and there exists a constant $c_\omega>0.09$ such that we have for any $t\leq X$ and any character $\chi(q)$ with $q\leq Q$ that
\begin{align}\label{rhocondition}
\sum_{\substack{n\in (X-3Y,X-Y]\\ n\leq t }}\Bigl(\omega(n)\chi(n)-c_\omega \delta_\chi\Bigr)\ll Y(\log Y)^{-2A}.
\end{align}
\item The coefficients are bounded by $\nu(n)\ll d(n)^{O(1)}\log n$ and $\omega(n)\ll d(n)^{O(1)}\log n$, where $d(n)$ denotes the number of divisors of $n$. 
\end{itemize}
\end{lemma}
\begin{proof}
Our choice of $\nu$ is based on Harman's results in section 10.6 \cite{harb}. More precisely, the process is based on the description on page 198 \cite{harb} in which by successive Buchstab iterations and discarding of parts a minorant for the primes is constructed. We obtain $\nu$ by multiplying it by $\log Y$. Our different choice of $Y$ does not affect his construction, since it only depends on the relative size of $H$ and $Y$ that remains unchanged. In this way, we get the character sum bound \eqref{nushort1} in the same way as Harman, see his condition in equation (10.2.8). Furthermore Harman's Buchstab constructions produce coefficients that are bound by some power of the divisor function. So we have 
\begin{align*}
\nu(n)\ll d(n)^{O(1)}\log n
\end{align*}
and so also get \eqref{nuest} for some finite $B$. We remark that Harman does not state the condition that $\nu$ should be supported on $H^{1/2}$-rough numbers only. However, he implicitly uses it and indeed it follows from the construction on page 198 by observing that the variable $z=Y^{1/9}$ appearing there is larger than $H^{1/2}=Y^{1/18}$.

The existence of $\omega$ is implied in Harman Theorem 10.8. To be more precise, we use Harman's construction of a minorant for the primes described in sections 7.4 to 7.9 in \cite{harb}. As interval we use $(X-3Y,X]$ and set $\omega$ to be $\log (X-3Y)$ times Harman's function. His construction starts with the Buchstab decomposition given in (7.4.2) and is completed by some intricate further decompositions for $\sum_3$ given throughout section 7.9. The coefficient bound is clear by the use of finitely many Buchstab decompositions. Finally condition \eqref{rhocondition} follows, since we can introduce characters by replacing Lemma 7.7 by Lemma 10.31 in section 7.5. In contrast to the construction of $\nu$, this does not incur any loss of constants, since $q\leq (\log Y)^A$. The introduction of characters is done in the same way by Harman in section 10.5, see the condition in equation (10.2.6). 
\end{proof}

With this, we have gathered all necessary tools for the final proof.

\begin{proof}[Proof of Theorem \ref{MT}]
Let $X$ be large, $Y$ and $H$ as in Lemma \ref{importlemma}. Note that this means for some $\epsilon'>0$ that $H=Y^{1/9+2\epsilon}=X^{7/120+\epsilon'}$, as stated in the theorem. 

We apply proposition \ref{prop} with
\begin{align*}
a(n)=b(n)=\Lambda'(n)
\end{align*}
and $\omega, \nu$ as given by Lemma \ref{importlemma}. By Lemma \ref{sievetolambdaq} and Corollary \ref{cor} the Fourier closeness conditions \eqref{propeq3} and \eqref{propeq4} hold with $\theta=(\log Y)^{-A-1}$. We can use the coefficient bounds for $\omega$ and $\nu$ to bound
\begin{align*}
\|\omega\|_2^2&\ll Y (\log Y)^{O(1)}\\
\|\nu\|_2^2 &\ll Y (\log Y)^{O(1)}.
\end{align*}
Similar bounds for $\|\mathcal{T}_\nu\|_2^2$ and $\|\mathcal{T}_\nu^+\|_2^2$ follow from their construction.

We choose $\kappa=Y/\log Y$. For any $A'>0$, after choosing $A$ sufficiently large, we obtain the following. We have for all $n\in [X-H,X]$ with at most
\begin{align*}
O\bigl(H(\log X)^{-A'}\bigr)
\end{align*}
exceptions that
\begin{align*}
\Lambda'*\Lambda'(n)\geq \omega*\mathcal{T}_\nu(n)+O(\frac{Y}{\log Y}).
\end{align*}

We are left with showing that $\omega*\mathcal{T}_\nu(n)\gg Y$ for even $n\in [X-H,X]$. We have
\begin{align*}
\omega*\mathcal{T}_\nu(n)&=\sum_{n_1\leq n}\omega(n_1)\mathcal{T}_\nu(n-n_1) \\
&=c_\nu  \sum_{n_1\in [n-Y,n-2Y]}\sum_{q\leq Q}\frac{\mu(q)}{\varphi(q)}\sum_{a(q)^*}e\bigl(\frac{a(n-n_1)}{q}\bigr)\omega(n_1)\\
&=\sum_{q\leq Q}\frac{\mu(q)}{\varphi(q)}\sum_{a(q)^*}e\bigl(\frac{an}{q}\bigr)\sum_{n_1\in [n-Y,n-2Y]}e\bigl(\frac{-an_1}{q}\bigr)\omega(n_1).
\end{align*}
Since $\omega$ is supported on $Q$-rough numbers only, we can apply the same Ramanujan sum strategy as in the proof of Lemma \ref{shorttoLq}. We have
\begin{align*}
\sum_{n_1\in [n-Y,n-2Y]}e\bigl(\frac{-an_1}{q}\bigr)\omega(n_1)=\frac{\mu(q)}{\varphi(q)}c_\omega Y+ E(n,q,a),
\end{align*}
where 
\begin{align*}
E(n,q,a)&=\sum_{\chi(q)} \frac{\tau(\overline{\chi}(q))\chi(-a)}{\varphi(q)}\sum_{n_1\in [n-Y,n-2Y]}\bigl(\omega(n)\chi(n_1)-\delta_\chi c_\omega\bigr).
\end{align*}
By condition \eqref{rhocondition} we get
\begin{align*}
\sum_{\chi(q)} \frac{\tau(\overline{\chi}(q))\chi(-a)}{\varphi(q)}\sum_{n_1\in [n-Y,n-2Y]}\bigl(\omega(n)\chi(n_1)-\delta_\chi c_\omega\bigr)
&\leq \sqrt{q} \max_\chi \bigl|\sum_{n_1\in [n-Y,n-2Y]}\bigl(\omega(n)\chi(n_1)-\delta_\chi c_\omega\bigr)| \\
&\ll Y(\log Y)^{-A/2}.
\end{align*}
Consequently, if $A$ is large enough in terms of $A'$ we have
\begin{align*}
\omega*\mathcal{T}_\nu(n)&= c_\omega c_\nu Y \sum_{q\leq Q}\frac{|\mu(q)|}{\varphi(q)^2}\sum_{a(q)^*}e\bigl(\frac{an}{q}\bigr)+O(Y(\log Y)^{-A'}).
\end{align*}
The singular series can be completed in standard fashion (see e.g. page 208 \cite{harb}) and we get outside of a sufficiently small exceptional set,
\begin{align*}
\omega*\mathcal{T}_\nu(n)&= c_\omega c_\nu  Y \mathfrak{S}(n)+O(Y(\log Y)^{-A'}),
\end{align*}
where
\begin{align*}
\mathfrak{S}(n)=\sum_{q=1}^\infty \frac{|\mu(q)|}{\varphi(q)^2}\sum_{a(q)^*}e\bigl(\frac{an}{q}\bigr)
\end{align*}
and for even $n$ we have uniformly $\mathfrak{S}(n)\gg 1$. This completes the proof.
\end{proof}

\Addresses

\end{document}